\documentclass[10pt]{amsart}
\usepackage{lmodern,xypic,somedefs,tracefnt,latexsym,rawfonts,graphicx,amsfonts,amssymb,latexsym,enumerate,amscd,hyphenat,amsthm}
\makeatletter
\@namedef{subjclassname@2020}{%
  \textup{2020} Mathematics Subject Classification}
\makeatother
\def\cal{\mathcal}
\def\Bbb{\mathbb}

\def\r{\rangle}
\def\l{\langle}

\def\icq{$c$\hyp{}quasifibration}

\def\kaq{$k$\hyp{}almost\hyp{}quasifibration}
\def\kcq{$k$\hyp{}$c$\hyp{}quasifibration}
\def\icd{{}$c$\hyp{}distinguished}
\def\kcd{$k$\hyp{}$c$\hyp{}distinguished}

\def\kaqs{$k$\hyp{}almost\hyp{}quasifibrations}
\def\kcqs{$k$\hyp{}$c$\hyp{}quasifibrations}
\def\icd{$c$\hyp{}distinguished}
\def\kcd{$k$\hyp{}$c$\hyp{}distinguished}

\def\cwe{$c$\hyp{}weak equivalence}
\def\wt{\widetilde}

\def\peq{\preccurlyeq}
\newtheorem{thm}{Theorem}[section]
\newtheorem{prop}[thm]{Proposition}
\newtheorem{exm}[thm]{Example}
\newtheorem{lemma}[thm]{Lemma}
\newtheorem{cor}[thm]{Corollary}
\newtheorem{defn}[thm]{Definition}
\newtheorem{rem}[thm]{Remark}

\newtheorem*{maint}{Main Theorem}
\numberwithin{equation}{section}
\begin{document}
\date{\today}
\title[$k$\hyp{}almost\hyp{}quasifibrations]
{$k$\hyp{}almost\hyp{}quasifibrations}
\author[S.K. Roushon]{S.K. Roushon}
\address{School of Mathematics\\
Tata Institute\\
Homi Bhabha Road\\
Mumbai 400005, India}
\email{roushon@math.tifr.res.in} 
\urladdr{http://www.math.tifr.res.in/\~\!\!\! roushon/}
\begin{abstract} In \cite{Rou22} we introduced the notion of a 
  $k$\hyp{}almost\hyp{}quasifibration. In this article we update this definition
  and call it a $k$\hyp{}$c$\hyp{}{\it quasifibration}. This will  
  help us to relate it to quasifibrations. We
  study some basic properties of \kcqs.
  We also generalize a series of results on quasifibrations (\cite{DT58}) to
  \kcqs\ giving criteria for a map to be a \kcq.
  \end{abstract}
 

\maketitle

\section{Introduction}
Recall that a surjective map $f:X\to Y$ is called a {\it quasifibration}
(\cite{DT58}, [\cite{Hat03}, chap 4, p. 479]) if for all $y\in Y$ and
$x\in F_y$, the map $f:(X, F_y, x)\to (Y,y)$ is a weak 
equivalence. Hence, a quasifibration $f:X\to Y$ induces a long exact sequence of homotopy
groups for all $y\in Y$.
In \cite{Rou22} we introduced the notion of a \kaq\  
for path connected spaces. For $k=\infty$, we called it an
{\it almost\hyp{}quasifibration}. It says that for some
$y\in Y$ there exist an exact sequence of homotopy groups of the above type. Hence
a quasifibration $f:X\to Y$ is an almost\hyp{}quasifibration.

The main motivation behind the definition of a \kaq\ was that
for computational purposes of homotopy groups, we need the long exact
sequence of homotopy groups induced by $f$, for some $y\in Y$, instead
of for all the points of $Y$.
Also, since we constructed a class of examples
in \cite{Rou22} supporting this definition. Furthermore, we had 
given many examples in \cite{Rou22} of
$1$\hyp{}almost\hyp{}quasifibrations and
almost\hyp{}quasifibrations 
which are not quasifibrations. 

In this article we make a general definition of \kaq,\ 
and add an extra condition. We call it
a {\it \kcq}, and for $k=\infty$, we call it a {\it \icq}. This extra
condition helps us to make the concept functorial
and to give a necessary and sufficient condition
for an almost\hyp{}quasifibration to be
a quasifibration. We also observe here that the examples of \kaqs\ we gave in 
\cite{Rou22} are all \kcqs .

Here, we plan to study some basic properties of
\kcqs . We prove several criteria for a
map to be a \kcq . These results are analogous to the fundamental results proved
in \cite{DT58} in the
context of quasifibrations. The methods used in the proofs of our
results are not new, but we see that they are applicable in the case
of \kcqs\ also.

Throughout the paper for topological spaces $X$ and $Y$,
$f:X\to Y$ will always denote a surjective continuous map.

\begin{defn}{\rm A subset of $Y$ is called {\it exhaustive} if it has a nonempty
intersection with each path component of $Y$.}\end{defn}

\begin{defn}\label{AQ}{\rm Let $k\geq 0$ be an integer. Choose
    an exhaustive subset $\wt Y$ of $Y$. 
     Then, $f$ is called
     a $k$-{\it almost\hyp{}quasifibration with respect to} $\wt Y$,
     if for all $y\in \wt Y$ and $x\in F_y:=f^{-1}(y)$
     there are homomorphisms

     $$\partial:\pi_{q+1}(Y, y)\to \pi_q(F_y, x),$$ for $q=0,1,2,\ldots , k-1$,
     so that the following sequence is exact.
     
\centerline{
      \xymatrix@C-=0.7cm{1\ar[r]&\pi_k(F_y, x)\ar[r]^{\!\!\!\!
          i_*}&
        \pi_k(X, x)\ar[r]^{\!\!\! f_*}&
        \pi_k(Y, y)\ar[r]^{\!\!\!\! \partial}&
        \pi_{k-1}(F_y, x)\ar[r]&
        \cdots}}
\centerline{
  \xymatrix@C-=0.7cm{\cdots\ar[r]&\pi_1(Y, y)\ar[r]^{\partial}&
    \pi_0(F_y, x)\ar[r]^{\!\!\!\!
      i_*}&\pi_0(X, x)\ar[r]^{\!\!\! f_*}&
    \pi_0(Y, y)\ar[r]&1.}}

\noindent
$f$ is called an {\it almost\hyp{}quasifibration}
if the above sequence can be extended on the
left up to infinity.}\end{defn}

Denote the path components
    of $Y$ by  $C_{\alpha}, \alpha\in I$.
    Let $X_{\alpha}=f^{-1}(C_{\alpha})$.
It is clear that $f$ is a $k$-almost\hyp{}quasifibration with respect
to an exhaustive subset $\wt Y$ if and only if
for all $\alpha\in I$, $f|_{{X_\alpha}}:X_{\alpha}\to C_{\alpha}$ is a
$k$-almost\hyp{}quasifibration with respect to $\wt Y\cap C_{\alpha}$.

In Definition \ref{AQ} we only assumed the existence of the connecting
homomorphisms $\partial$. In Lemma \ref{NASC} we observe
that the commutativity of 
Diagram 1 is the key connection between an almost\hyp{}quasifibration and a
quasifibration. In the diagram, $\partial_*$ is the connecting homomorphism coming from
the long exact sequence of homotopy groups for the pair $(X, F_y)$. The commutativity
of Diagram 1 makes $f_*$ a bijection and hence $\partial$ becomes unique.

\centerline{
       \xymatrix{&\ar[ld]_{f_*}\pi_{q+1}(X, F_y, x)\ar[d]^{\partial_*}\\
         \pi_{q+1}(Y,y)\ar[r]^{\partial}&\pi_q(F_y, x).}}

     \centerline{Diagram 1}
     
    Therefore, we now make the following definition. 

   \begin{defn}\label{KQ}{\rm A $k$-almost\hyp{}quasifibration $f:X\to Y$
       with respect to an exhaustive subset $\wt Y \subset Y$ is called a
       $k$\hyp{}$c$\hyp{}{\it quasifibration} with respect to
       $\wt Y$, if Diagram 1 
       is commutative for $q=0,1,2,\ldots , k-1$, for all $y\in
       \wt Y$ and $x\in F_y$. When $k=\infty$ we call $f$ a {\it \icq}.}\end{defn}
   
  Therefore, we observe that a \icq\ $f:X\to Y$
     with respect to $Y$ is a genuine quasifibration in the sense of
     \cite{DT58} (Corollary \ref{relation}). And the fibers of a
     \icq\ with respect to $\wt Y$, over the points of $\wt Y$ are all
     weak homotopy equivalent.
     We also show in Lemma \ref{connection} that the examples of
     $k$\hyp{}almost\hyp{}quasifibrations, (for $k=1, \infty$) we
     gave in \cite{Rou22} are all $k$\hyp$c$\hyp{}quasifibrations
     (see Example \ref{example}).

     In the next section we state our main results. In Section 3
     we relate quasifibrations and \kcqs\ and prove some basic results
     we need. Section 4 contains the proofs
     of the main results.
     
\section{Statements of main results}

Recall that for a surjective map $f:X\to Y$, a subset $Y_1\subset Y$ is called
distinguished (\cite{DT58}) if $f|_{f^{-1}(Y_1)}:f^{-1}(Y_1)\to Y_1$ is a
quasifibration. We need the following analogue of this definition, in the
context of \kcqs.

\begin{defn}\label{dist} {\rm A subset $Y_1\subset Y$ is called
     ($k$\hyp{})$c$\hyp{}{\it distinguished} with respect to an exhaustive
     subset $\wt Y_1\subset Y_1$, if $f|_{f^{-1}(Y_1)}:f^{-1}(Y_1)\to Y_1$ is a 
     ($k$\hyp{})$c$\hyp{}quasifibration with respect to
     $\wt Y_1$.}\end{defn}
 
 We begin with the following result giving a criterion for a map
 $f$ to be a \icq\ from local data.
 
\begin{thm}\label{covering} Let $\{U_{\alpha}\}_{\alpha\in J}$
  be an open covering of $Y$, which is closed under taking finite intersections.
  Assume that, for any $\alpha\in J$, there is an exhaustive subset 
  $\wt U_{\alpha}\subset U_{\alpha}$, such that  $\wt
  U_{\beta}\subset \wt U_{\gamma}$ whenever   $U_{\beta}\subset
  U_{\gamma}$, for $\beta, \gamma\in J$. Furthermore, assume that $U_{\alpha}$ is \icd\  
  with respect to $\wt U_{\alpha}$, for any $\alpha\in J$.
  Then $f$ is a \icq\ with respect to some exhaustive subset of
  $Y$.\end{thm}

 Now, recall that a {\it filtration} of a space $Y$ is an increasing 
 sequence of subspaces $Y_0\subset Y_1\subset\cdots$ such that
 $Y=\cup_{i=0}^{i=\infty}Y_i$, and $Y$ has the colimit topology.

 For each $i\in {\Bbb N}$, let $\wt Y_i\subset Y_i$ be an exhaustive subset of $Y_i$, such
 that $\wt Y_i\subset \wt Y_{i+1}$. Then clearly, the union
 $\wt Y:= \cup _{i\in {\Bbb N}}\wt Y_i$ is an exhaustive subset of
 $Y$. We call $\wt Y$ an {\it exhaustive subset of $Y$ for 
   the filtration} $\{Y_i\}_{i\in {\Bbb N}}$.

 More generally, for any
 covering $\{Y_{\alpha}\}_{\alpha\in J}$ of $Y$ by subsets with
 exhaustive subsets $\wt Y_{\alpha}\subset Y_{\alpha}$, for $\alpha\in J$,
 $\wt Y:=\cup_{\alpha\in J}\wt Y_{\alpha}$ is an exhaustive subset of $Y$.

 The following result shows that being a 
 $k$\hyp{}$c$\hyp{}quasifibration
 is preserved under taking colimit.

 \begin{thm}\label{limit} Let $Y$, $Y_i$ and $\wt Y_i$ be as above.
   Assume that, for each $i$, $Y_i$ is $T_1$ and \kcd\ 
   with respect to the exhaustive subset $\wt Y_i$. Then, $f$ is a 
   \kcq\ with respect to $\wt Y$.\end{thm}

 In the next result we show that being a $k$\hyp{}$c$\hyp{}quasifibration is
 also preserved under deformation. For this, we need the
 following two definitions.

First, we make a variation
of the definition of a $k$-equivalence.

\begin{defn}\label{kcequiv} {\rm For pairs $(X,X_1)$ and $(Y,Y_1)$ of topological spaces, 
  a map $g:(X,X_1)\to (Y,Y_1)$ is called a
  $k$\hyp{}$c$\hyp{}{\it equivalence} with respect to a subset
  $\wt X_1\subset X_1$, if the following conditions are satisfied.

$\bullet$ $g_*^{-1}Image(\pi_0(Y_1)\to \pi_0(Y))=Image(\pi_0(X_1)\to \pi_0(X))$.

$\bullet$ For all $x\in \wt X_1$,  $g_*:\pi_q(X,X_1,x)\to \pi_q(Y,Y_1,g(x))$ is a bijection for
$1\leq q\leq k-1$, and is a surjection for $q=k$.

$g$ is called a $c$\hyp{}{\it weak equivalence} with respect to $\wt X_1$,
if the above conditions are satisfied for
all $k$.}\end{defn}

We need this variation, since
we will be considering only selected fibers of $f:X\to Y$ in this
work. Here, notice that a $k$\hyp{}$c$\hyp{}equivalence is defined with
respect to a subset of the domain, and a \kcq\ is defined
with respect to a subset of the codomain.

\begin{rem}\label{kce}{\rm Recall that a $k$-equivalence demands
    that the second condition in Definition \ref{kcequiv} must be satisfied
for all $x\in X_1$. If $\wt X_1$ is an exhaustive subset of $X_1$, then it
has nonempty intersection with each path component
of $X_1$, and hence a $k$\hyp{}$c$\hyp{}equivalence with respect to an exhaustive subset,
is a $k$\hyp{}equivalence.
Similarly a $c$\hyp{}weak equivalence with respect to an exhaustive subset 
is a weak equivalence in the standard sense.}\end{rem}

 \begin{defn}\label{deformation}{\rm  Let $Y_1\subset Y$ be a subset.
     Let $X_1=f^{-1}(Y_1)$. A {\it deformation}
 of $f:X\to Y$ to $f_1:=f|_{X_1}:X_1\to Y_1$ is a pair $(H, h)$ of maps defined by  
 $H:X\times I\to X$ and $h:Y\times I\to Y$, such
 that the following conditions are satisfied.

 $\bullet$ $h|_{Y\times 0}=id_Y$, $h_t(y_1):=h(t,y_1)=y_1$ for all $y_1\in Y_1$,
 $t\in I$ and $h_1(y)\in Y_1$ for all $y\in Y$.

 $\bullet$ $H|_{X\times 0}=id_X$, $H_t(x_1):=H(t,x_1)=x_1$ for all $x_1\in X_1$,
 $t\in I$ and $H_1(x)\in X_1$ for all $x\in X$.

 $\bullet$ $f\circ H_1=h_1\circ f_1$.}\end{defn}

We are now in a position to state our next result.

\begin{thm}\label{deformationthm} Let $Y_1\subset Y$ be a subset and
     $\wt Y\subset Y$, $\wt Y_1\subset Y_1$ be 
     exhaustive subsets, such that $\wt Y_1\subset \wt Y$.
     Then $f:X\to Y$ is a  
     \icq\ with respect to $\wt Y$ if the following are satisfied.

     $\bullet$ There is a deformation $(H, h)$ from $f$ to $f_1$ such
     that $h_1(\wt Y)\subset \wt Y_1$.

$\bullet$ $H_1:f^{-1}(y)\to f^{-1}(h_1(y))$ is a weak equivalence for all $y\in \wt Y$.

$\bullet$ $Y_1$ is \icd\ with respect to $\wt Y_1$.
\end{thm}
    
Combining Theorems \ref{covering}, \ref{limit}  and \ref{deformationthm},
we can now deduce our final result. This is analogous to Theorem 2.6
in \cite{May88}.

\begin{maint}\label{maintheorem} Let $\{Y_i\}_{i\in {\Bbb N}}$ be a
  filtration of $Y$ by closed and $T_1$ subspaces.
  Let, $\wt Y=\cup_{i\in {\Bbb N}}\wt Y_i$ be an exhaustive subset 
  of $Y$ for the filtration
  $\{Y_i\}_{i\in {\Bbb N}}$. Assume that $\wt Y_i$ intersects each
  open subset $U$ of $Y_i$
  in an exhaustive subset $\wt U$ (e.g., if $\wt Y_i$ is dense in
  $Y_i$), and the following is satisfied.

  $\bullet$ $Y_1$ is \icd\ with respect to
  $\wt Y_1$ and for each $i\geq 1$, each open subset $U$ of $Y_{i+1}-Y_i$ is
  $c$\hyp{}distinguished with respect to the exhaustive
  subset $\wt U\subset \wt Y_{i+1}$. For
  each $i\in {\Bbb N}$, $Y_i$ has a neighborhood $U_{i+1}$ in $Y_{i+1}$ and a
  deformation $(H,h)$ from $f|_{f^{-1}(U_{i+1})}:f^{-1}(U_{i+1})\to U_{i+1}$ to
  $f|_{f^{-1}(Y_i)}:f^{-1}(Y_i)\to Y_i$, such that $h_1(\wt U_{i+1})\subset \wt Y_{i+1}$
  and $H_1:f^{-1}(y)\to f^{-1}(h_1(y))$
  is a weak equivalence for all $y\in \wt U_{i+1}$.

  Then, each $Y_i$ is \icd\ with respect to
  $\wt Y_i$ and $f$ is a  \icq\ with respect to
  $\wt Y$.\end{maint}

Here, we recall that the quasifibration versions of
Theorems \ref{covering}, \ref{limit} and \ref{deformationthm} were the
basic tools behind the proof of the
Dold-Thom Theorem (\cite{DT58}). In particular, using these theorems
one shows that for a based connected $CW$-pair $(X, A)$, the map
$SP(X)\to SP(X/A)$ is a quasifibration with fiber $SP(A)$. Here,
$SP(-)$ is the union of the $n$-fold symmetric products 
$SP_n(-)$. Recall that, the functor $SP_n(-)$ sends a space to the
quotient of its $n$-fold product by the obvious
action of the symmetric group $S_n$. Also
see [\cite{Hat03}, p. 484]. 

\section{Quasifibrations and \kcqs\ }
In this section we study properties of \kcqs\ and prove some
basic results needed for the proofs of the theorems.

\subsection{\kcqs\ and its functoriality}

In the following lemma we give a simple necessary and sufficient condition for an
almost\hyp{}quasifibration to be a quasifibration.

\begin{lemma}\label{NASC} $f:X\to Y$ is a quasifibration if and only if $f$ is an
  almost\hyp{}quasifibration with respect to $Y$, and Diagram 1 is commutative
  for all $y\in Y$, $x\in F_y$ and for $q=0, 1,2,\ldots$.\end{lemma}

   \begin{proof} We observe that if $f$ is an almost\hyp{}quasifibration and
     Diagram 1 is commutative then, Diagram 2 is also commutative. 
     
     Hence using the Five Lemma, we see that 
     $f_*:\pi_{q+1}(X, F_y, x)\to \pi_{q+1}(Y,y)$ is a bijection.
     Here $j:(X,x)\to (X, F_y, x)$ is the
     inclusion map.
     
     \centerline{
       \xymatrix{\pi_{q+1}(F_y, x)\ar[r]^{i_*}\ar[d]^{=}&\pi_{q+1}(X,
         x)\ar[r]^{j_*}\ar[d]^{=}&
         \pi_{q+1}(X, F_y, x)\ar[r]^{\partial_*}\ar[d]^{f_*}&
         \pi_q(F_y, x)\ar[r]^{i_*}\ar[d]^{=}&\pi_q(X, x)\ar[d]^{=}\\
         \pi_{q+1}(F_y, x)\ar[r]^{i_*}&\pi_{q+1}(X, x)\ar[r]^{f_*}&\pi_{q+1}(Y, y)\ar[r]^{\partial}&
         \pi_q(F_y, x)\ar[r]^{i_*}&\pi_q(X, x)}}

     \centerline{Diagram 2}
     
     Conversely, assume that $f$ is a quasifibration, and hence
     $f_*:\pi_{q+1}(X, F_y, x)\to \pi_{q+1}(Y,y)$
     is a bijection. 
     Then, we can define $\partial$ as $\partial_*\circ f_*^{-1}$,
     which makes $f$ an almost\hyp{}quasifibration
     and Diagram 1 is commutative.

     This proves the lemma.
   \end{proof}

   \begin{cor} \label{relation} $f:X\to Y$ is a \icq\ with respect to
     $Y$ if and only if $f$ is a quasifibration.\end{cor}

   Therefore, if $f$ is a \icq\ with respect to an exhaustive subset $\wt Y$, and 
   $\wt Y$ is \icd\ with respect to $\wt Y$, then $f|_{f^{-1}(\wt Y)}$ is a quasifibration.
   
   \begin{lemma} \label{connection} A $1$\hyp{}almost\hyp{}quasifibration with respect
     to some exhaustive subset $\wt Y$ is a
     $1$\hyp{}$c$\hyp{}quasifibration with respect to
     $\wt Y$, if $i_*:\pi_0(F_y, x)\to \pi_0(X,x)$ is an injection,
     equivalently if $f_*:\pi_1(X,x)\to\pi_1(Y,y)$ is a surjection, for
     all $y\in \wt Y$ and $x\in F_y$.\end{lemma}

   \begin{proof} The proof is clear from Diagram 3.
     Since we only have to show that the triangle is commutative. But that  
follows, as $i_*$ is an injection if and only if $\partial$ is trivial
and also $i_*$ is an injection if and only if $\partial_*$ is trivial.\end{proof}

     \centerline{
       \xymatrix@C-=0.5cm{&&&&
         \pi_1(X,F_y,x)\ar[ld]_{f_*}\ar[d]^{\partial_*}&&\\
         1\ar[r]&\pi_1(F_y,x)\ar[r]^{i_*}&\pi_1(X,x)\ar[r]^{f_*}&\pi_1(Y,y)\ar[r]^{\partial}&
         \pi_0(F_y,x)\ar[r]^{i_*}&
         \pi_0(X,x)\ar[r]^{\ \ \ \ f_*}&\cdots.}}

     \centerline{Diagram 3}

     Therefore, we see that if $F_y$ is connected
     for all $y\in \wt Y$, then there is no difference between a
     $1$\hyp{}almost\hyp{}quasifibration and a 
     $1$\hyp{}c\hyp{}quasifibration with respect to $\wt Y$.

   More generally, we have the following criterion.
   
   \begin{lemma}\label{criterio} 
     $f:X\to Y$ is a $k$\hyp{}$c$\hyp{}quasifibration  
     with respect to an exhaustive subset $\wt Y$ if and 
     only if $j_*:\pi_{k+1}(X, x)\to \pi_{k+1}(X,F_y,x)$ is a surjection and
     the map $f_*:\pi_{q+1}(X, F_y, x)\to \pi_{q+1}(Y, y)$ is 
     a bijection, for
     all $y\in \wt Y$, $x\in F_y$ and for $q=0,1,2\ldots, k-1$.\end{lemma}

   \begin{proof}  Consider the general Diagram 4.
     
     First note that $j_*:\pi_{k+1}(X, x)\to \pi_{k+1}(X,F_y,x)$ is a
     surjection if and only if $\partial_*:\pi_{k+1}(X, F_y, x)\to \pi_k(F_y, x)$ is
     the trivial homomorphism if and only if $i_*:\pi_k(F_y, x)\to \pi_k(X, x)$
     is an injection.
     
\centerline{
       \xymatrix@C-=0.75cm{&&&&\pi_{k+1}(X,
         F_y,x)\ar[ld]^{\ f_*}\ar[d]^{\partial_*}&\\
        \cdots\ar[r]&\pi_{k+1}(F_y, x)\ar[r]^{i_*}&\pi_{k+1}(X, x)\ar[r]^{\ \ \ f_*}\ar[rru]^{j_*}&
        \pi_k(Y, y)\ar[r]^{\partial}&
         \pi_k(F_y, x)\ar[r]^{\ \ \ \ \ \ \ i_*}&\cdots}}
\centerline{
       \xymatrix{&&\pi_1(X,F_y,x)\ar[ld]_{\!\!\!\! f_*}\ar[d]^{\partial_*}&&&\\
    \cdots\ar[r]^{\!\!\!\!\!\!\!\!\! f_*}&
         \pi_1(Y, y)\ar[r]^{\partial}&
         \pi_0(F_y, x)\ar[r]^{i_*}&
         \pi_0(X, x)\ar[r]^{\! \! \!  f_*}&\pi_0(Y, y)\ar[r]&1.}}

     \centerline{Diagram 4}
     
     If $f$ is a $k$\hyp{}$c$\hyp{}quasifibration then
     $i_*:\pi_k(F_y, x)\to \pi_k(X,x)$ is an injection by
     definition. Also since Diagram 1 is
     commutative for all $q=0,1,2,\ldots, k-1$, by
     the Five Lemma applied to Diagram 2, we get 
     $f_*:\pi_{q+1}(X, F_y, x)\to \pi_{q+1}(Y, y)$ is a bijection.

     Conversely, assume that $i_*:\pi_k(F_y, x)\to
     \pi_k(X,x)$ is an injection and
     $f_*:\pi_{q+1}((X, F_y, x))\to \pi_{q+1}((Y,
     y))$ is a bijection for $q=0,1,2\ldots, k-1$. Once again, as in
     Lemma \ref{NASC} we can define the homomorphisms $\partial$, so as
     to get the
     required exact sequence and Diagram 1 to be commutative.\end{proof}

   \begin{cor}\label{aspherical} Assume $\pi_2(X, x)=\l 1\r$ for all $x\in X$. Then $f$ is a
     $1$\hyp{}$c$\hyp{}quasifibration with respect to $\wt Y$,
     if and only if $\pi_2(X, F_y, x)=\l 1\r$ and
     $f_*:\pi_1(X, F_y, x)\to \pi_1(Y, y)$ is a bijection, for all $y\in \wt Y$ and 
     $x\in F_y$.\end{cor}

   We now give some non-trivial examples of \kcqs.

   \begin{exm}\label{example}{\rm Let $S$ be a connected aspherical
       $2$-manifold. Let $G$ be a
       discrete group acting on $S$, effectively and properly discontinuously
       with isolated fixed points. If $S/G$ has genus zero assume that
       it has a puncture. Consider the space ${\cal O}_n(S,G)$ of
       $n$-tuples 
       of points of $S$ with pairwise distinct orbits. Then, in
       [\cite{Rou22}, Theorem 1.3 and Proposition 1.6] we had shown that
       the projection $p_l:{\cal O}_n(S,G)\to {\cal O}_l(S,G)$ to the
       first $l$ coordinates is a $1$-almost\hyp{}quasifibration, with
       respect to the subset $\wt {{\cal O}_l(S, G)}$ of
       points of ${\cal O}_l(S, G)$, whose 
       coordinates have trivial isotropy groups.
       And the projection is an almost\hyp{}quasifibration 
       with respect to $\wt {{\cal O}_l(S, G)}$, if either the action
       of $G$ on $S$ is free, or if $S/G$ has genus zero with at the
       most
      two cone points of order $2$ each. Using Lemma \ref{connection},
      now it follows that $p_l$
       is a $1$\hyp{}$c$\hyp{}quasifibration 
       or a $c$\hyp{}quasifibration with respect to $\wt {{\cal O}_l(S, G)}$, 
       for the respective cases as above.}\end{exm}

   The next examples are from $3$-manifold topology and of much importance, since
   they are the building blocks of $3$-manifolds, other than the hyperbolic
   ones. The standard references for the background of these examples are \cite{Hem76}
    and \cite{Sco83}.
   
\begin{exm}\label{example2}{\rm 
    Let $M$ be a connected Seifert fibered space. Let $B$ be the base orbifold 
    with infinite orbifold fundamental group. Let $q:M\to B$ be
    the quotient map. Then, it is well known that for any regular 
    point $b\in B$, there is the following exact
    sequence (\cite{Sco83}, Lemma 3.2). Here,
    ${\Bbb S}^1_b$ is the circle fiber over $b$ and $m\in {\Bbb S}^1_b$.

    \centerline{
      \xymatrix{1\ar[r]&{\Bbb Z}\simeq\pi_1({\Bbb S}^1_b, m)\ar[r]&\pi_1(M, m)\ar[r]&\pi_1^{orb}(B, b)\ar[r]&1.}}

    Note that $B$ is a good
    orbifold, since it has infinite orbifold fundamental group (\cite{Rou21}, Proposition 3.15).
    That is, there is a
  $2$-manifold $\hat B$ and a discrete group $G$ acting effectively and
  properly discontinuously on $\hat B$ with $B=\hat B/G$. Let $p:\hat B\to B$ be
  the orbifold covering map. Let $\hat p:\hat M\to \hat B$ be the pull back of
  $p$ by $q$ as shown in Diagram 5. Let $\hat G$ be the group of
  covering transformation of $\hat p$.

  \centerline{
    \xymatrix{\hat M\ar[r]^{\hat p}\ar[d]^{\hat q}&M\ar[d]^q\\
      \hat B\ar[r]^p&B}}

  \centerline{Diagram 5}
\noindent  
   For the rest of the deduction
  we follow the reference \cite{Moe02}. Let ${\cal G}(\hat M, \hat G)$ and
  ${\cal G}(\hat B, G)$ be the corresponding translation
  Lie groupoids.
  The above data induces a homomorphism
  $Q:{\cal G}(\hat M, \hat G)\to {\cal G}(\hat B, G)$, which in turn defines a
  continuous map $BQ:B{\cal G}(\hat M, \hat G)\to B{\cal G}(\hat B, G)$
  on their classifying spaces. From some generalities, we have the
  following two isomorphisms.
  
  $$\pi_1(B{\cal G}(\hat M, \hat G), \hat m)\simeq \pi_1(M,m)\ \text{and}\ \pi_1(B{\cal G}(\hat B, G), \hat b)\simeq
  \pi_1^{orb}(B,b).$$ Here, $b\in B$ is a regular point such that
  $p(\hat b)=b$, $\hat p(\hat m)=m$ and $q(m)=b$. Furthermore,
  $BQ$ induces the above exact sequence. Therefore, $BQ$ is a $1$\hyp{}$c$\hyp{}quasifibration
  with respect to the set $\wt {\hat B}$ of all points
  of $\hat B$ with trivial isotropy groups. Consequently, it is also a \icq\ 
with respect to $\wt {\hat B}$, since the above
  two classifying spaces are aspherical, which is a consequence of the fact that
  both $\hat M$ and $\hat B$ are aspherical.}
\end{exm}

The following examples, although easy, show that there are \kcqs\ for
   any $k$.
   
   \begin{exm}\label{example1}{\rm
       One can also 
       construct examples of \kcqs\ from Example \ref{example} by 
       taking product with a space $Z$, that is by considering  
$$p_l\times id_Z:{\cal O}_n(S,G)\times Z\to {\cal O}_l(S,G)\times
Z.$$ Then, $p_l\times id_Z$ becomes a $1$\hyp{}$c$\hyp{}quasifibration with
respect to $\wt {{\cal O}_l(S, G)}\times Z$. More generally, for a
\kcq\ $f$ with respect $\wt Y$, $f\times id_Z:X\times Z\to Y\times Z$
is a \kcq\ with respect to $\wt Y\times Z$. In addition, if $X$, $Y$
and $F_y$ are all aspherical for all $y\in \wt Y$ and if $f$ is a
$1$\hyp{}$c$\hyp{}quasifibration, then $f\times id_Z$ is a \kcq\ for
all $k$. Also, for any quasifibration $f$, if we change the sign
of $\partial$ at some stage, say $q$ and for all $y\in Y$,
then the long sequence of homotopy
groups still remains exact, and
hence $f$ has an almost-quasifibration structure which is not a 
\icq, since Diagram 1 at stage $q$ is not commutative.}\end{exm}

       We also asked in \cite{Rou22} if
       for any connected manifold $M$ with an effective and properly
       discontinuous action of a discrete group $G$, $p_l:{\cal O}_n(M,G)\to
       {\cal O}_l(M,G)$
       is an almost\hyp{}quasifibration (\icq)\ with respect
       to $\wt {{\cal O}_l(M, G)}$?

       Next, we observe that the exact sequence in the 
definition of \kcq\ is functorial, in the sense described in the following
lemma. The lemma is easy to verify.

\begin{lemma}\label{functorial}
Let $f:X\to Y$ and $\hat f:\hat X\to \hat Y$ be two \kcqs\ with respect
     to $\wt Y$ and $\wt {\hat Y}$, respectively. Assume that there
     are continuous maps $g:X\to \hat X$ and $h:Y\to \hat Y$ making 
     Diagram 6 commutative and that $h(\wt Y)\subset \wt {\hat Y}$.

     \centerline{
       \xymatrix{X\ar[r]^f\ar[d]^g&Y\ar[d]^h\\
         \hat X\ar[r]^{\hat f}&\hat Y}}

     \centerline{\rm{Diagram 6}}

     Then, for all $y\in \wt Y$ and $x\in F_y$, 
     Diagram 7 is commutative. Here, $\hat x=g(x)$, $\hat
     y=h(y)$ and $\hat F_{\hat y}={\hat f}^{-1}(\hat y)$.
   
     \centerline{
      \xymatrix@C-=0.5cm{1\ar[r]&\pi_k(F_y,x)\ar[r]^{i_*}\ar[d]^{g_*}&
        \pi_k(X,x)\ar[r]^{f_*}\ar[d]^{g_*}&
        \pi_k(Y,y)\ar[r]^{\!\!\!\!\!\! \partial}\ar[d]^{h_*}&
        \pi_{k-1}(F_y,x)\ar[r]\ar[d]^{g_*}&
        \cdots\ar[r]^{\!\!\!\!\!\! \hat f_*}&\pi_0(Y,y)\ar[r]\ar[d]^{h_*}&1\\
        1\ar[r]&\pi_k(\hat F_{\hat y},\hat x)\ar[r]^{i_*}&
        \pi_k(\hat X,\hat x)\ar[r]^{\hat f_*}&
        \pi_k(\hat Y,\hat y)\ar[r]^{\!\!\!\!\!\! \partial}&
        \pi_{k-1}(\hat F_{\hat y},\hat
        x)\ar[r]&\cdots\ar[r]^{\!\!\!\!\!\! \hat f_*}&\pi_0(\hat Y,\hat y)\ar[r]&1}}

    \centerline{\rm{Diagram 7}}\end{lemma}
       
   \subsection{Some basic results}
   
   We will need the following local to global type lemma for the proofs of the theorems.

   \begin{lemma}\label{directlimit} Let ${\cal O}=\{Y_{\alpha}\}_{\alpha\in J}$ be a
     covering of $Y$ by subsets. Consider the partial ordering on $J$ induced by
     set inclusions of the members of
     $\cal O$, and let this order (denoted by $\peq$) be a total ordering. Assume that
     any compact subset of $Y$ is contained in some $Y_{\alpha}\in {\cal O}$. Let
     $\wt Y_{\alpha}$ be an exhaustive subset
     of $Y_{\alpha}\in {\cal O}$, with the property that
     if $\alpha\peq {\beta}$ in $J$, then $\wt Y_{\alpha}\subset \wt Y_{\beta}$. If
      each $Y_{\alpha}\in \cal O$ is  \kcd\ with respect to
      $\wt Y_{\alpha}$, then $f$ is a \kcq\ with respect
      to $\cup_{\alpha\in J}\wt Y_{\alpha}$.\end{lemma}

   \begin{proof} Let $y\in \cup_{\alpha\in J}\wt Y_{\alpha}$ and choose $\beta\in I$
     so that $y\in \wt Y_{\beta}$. Let $X_{\alpha}=f^{-1}(Y_{\alpha})$
     for all $\alpha\in J$.
     For $\beta\peq \gamma\peq \delta$ in $J$, consider Diagram 8. 

     \centerline{
       \xymatrix{&\ar[ld]_{f_*}\pi_{q+1}(X_{\gamma}, F_y, x)\ar[d]^{\partial_*}\ar[rd]&\\
 \pi_{q+1}(Y_{\gamma},y)\ar[r]^{\partial}\ar[rd]&\pi_q(F_y, x)\ar[rd]&
 \ar[ld]_{\ \ \ \ \ \ \ f_*}\pi_{q+1}(X_{\delta}, F_y, x)\ar[d]^{\partial_*}\\
         &\pi_{q+1}(Y_{\delta},y)\ar[r]^{\partial}&\pi_q(F_y, x)}}

     \centerline{Diagram 8}

     The two vertical triangles are commutative for $q=0,1,2,\ldots, k-1$,
     since each $f|_{X_{\alpha}}$ is a \kcq\ with respect to $\wt
     Y_{\alpha}$, and $\wt Y_{\beta}\subset \wt Y_{\gamma}\subset \wt Y_{\delta}$.
     And the whole diagram is commutative by functoriality of boundary maps.
     Now, note that, since any compact subset of $Y$ is contained in some member
     of $\cal O$, it is contained in $Y_{\alpha}$ for some
     $\alpha \in J$ with $\beta\peq \alpha$, as $(J, \peq)$ is totally
     ordered. Therefore,  
      a colimit argument on Diagram 8 gives the commutative Diagram 1, and hence we obtain
     a definition
     of $\partial$ for $f:X\to Y$.
     
   \centerline{
      \xymatrix@C-=0.7cm{1\ar[r]&\pi_k(F_y,x)\ar[r]^{\!\!\!\!
        i_*}\ar[d]&
      \pi_k(X_{\gamma},x)\ar[r]^{\!\!\! f_*}\ar[d]&
        \pi_k(Y_{\gamma},y)\ar[r]^{\!\!\!\! \partial}\ar[d]&\pi_{k-1}(F_y,x)\ar[r]\ar[d]&
        \cdots\\
     1\ar[r]&\pi_k(F_y,x)\ar[r]^{\!\!\!\!
          i_*}&
        \pi_k(X_{\delta},x)\ar[r]^{\!\!\! f_*}&
        \pi_k(Y_{\delta},y)\ar[r]^{\!\!\!\! \partial}&\pi_{k-1}(F_y,x)\ar[r]&
        \cdots}}
    
 \centerline{Diagram 9}

Next, consider Diagram 9. This diagram is commutative by functoriality
of \kcq\ (see Lemma \ref{functorial}). Hence, again 
     by taking colimit we see that, since $f|_{X_{\alpha}}$ is a \kcq\
     with respect to $\wt Y_{\alpha}$, for each $\alpha$,
     $f$ is a \kcq\ with respect
     to $\cup_{\alpha\in J}\wt Y_{\alpha}$.
  \end{proof}
   
     The following proposition is
   the main ingredient for the proof of Theorem \ref{covering}. This
   result is analogous to [\cite{DT58}, Satz 2.2].

   Let $Y_1, Y_2\subset Y$ be such that $Y$ is the union of the interiors of
   $Y_1$ and $Y_2$, then it is standard to call $(Y; Y_1, Y_2)$ an {\it excisive triad}. 
 
 \begin{prop}\label{coveringprop} Let $(Y; Y_1, Y_2)$ be an excisive triad. 
Let $\wt Y_{12}\subset Y_1\cap Y_2$, $\wt Y_1\subset Y_1$ and
$\wt Y_2\subset Y_2$ be exhaustive subsets such that $\wt Y_{12}\subset \wt Y_1\cap \wt Y_2$.
If $Y_1$, $Y_2$ and $Y_1\cap Y_2$ are all
   $c$\hyp{}{\it distinguished} with respect to the
   corresponding exhaustive subsets above, then $f$ is a 
   $c$\hyp{}quasifibration with respect to $\wt Y_1\cup \wt Y_2$.\end{prop}

 To prove the proposition we need the following two lemmas.

 \begin{lemma}\label{funda} Let $Y_1\subset Y$ be $c$\hyp{}distinguished
   with respect to an exhaustive subset $\wt Y_1\subset Y_1$. Then $f:(X, F_y, x)\to (Y, y)$
   is a weak equivalence for all  $y\in \wt Y_1$ if and only if
   $f:(X, f^{-1}(Y_1), x)\to (Y, Y_1, y)$ is a $c$\hyp{}weak equivalence
   with respect to $f^{-1}(\wt Y_1)$.\end{lemma}

 \begin{proof} For $y\in \wt Y_1$,
  consider the following map of triples 
  $$f:(X, f^{-1}(Y_1), f^{-1}(y))\to (Y, Y_1, y).$$ Then,
  the proof follows by applying the Five Lemma on the commutative
  diagram  of the long exact sequences 
  of homotopy groups for triples, induced by the above map of
  triples.\end{proof}

 The following lemma is analogous to 
  Proposition 4K.1 of \cite{Hat03}, for our variation of the
  definition of a $k$-equivalence.
  
 \begin{lemma}\label{n-equiv}  Let $f:(X; X_1, X_2)\to (Y; Y_1, Y_2)$ be a map
   of excisive triads. Let $\wt X_{12}\subset X_1\cap X_2$, $\wt X_1\subset X_1$ and
   $\wt X_2\subset X_2$ be exhaustive subsets. Assume that,  
$f:(X_i, X_1\cap X_2)\to (Y_i, Y_1\cap Y_2)$ is a $k$\hyp{}$c$\hyp{}equivalence 
with respect to $\wt X_{12}$, for $i=1,2$. Then $f:(X, X_i)\to (Y, Y_i)$ 
is a $k$\hyp{}$c$\hyp{}equivalence with respect to $\wt X_i$, 
for $i=1,2$.\end{lemma}

\begin{proof} First note that, since $\wt X_{12}$ is exhaustive
  in $X_1\cap X_2$, $f:(X_i, X_1\cap X_2)\to (Y_i, Y_1\cap Y_2)$ is a
  $k$\hyp{}equivalence, for $i=1,2$, by Remark \ref{kce}. Therefore,
  by Proposition 4K.1 of \cite{Hat03}, $f:(X, X_i)\to (Y, Y_i)$ 
  is a $k$\hyp{}equivalence,
  and hence, in particular, $k$\hyp{}$c$\hyp{}equivalence
with respect to $\wt X_i$, for $i=1,2$.\end{proof}

Now, we can prove Proposition \ref{coveringprop}.

\begin{proof}[Proof of Proposition \ref{coveringprop}] First note that, 
  since $Y_1\cap Y_2$ is \icd\ with respect to $\wt Y_{12}$, for $i=1,2$,
  $f:(f^{-1}(Y_i), F_y)\to (Y_i, y)$ is a weak equivalence for all $y\in \wt Y_{12}$.
  See Corollary \ref{relation}.
  Hence by Lemma \ref{funda}, for $i=1,2$,
  $$f:(f^{-1}(Y_i),f^{-1}(Y_1)\cap f^{-1}(Y_2))\to (Y_i, Y_1\cap Y_2)$$ is a
  \cwe\ with respect to $f^{-1}(\wt Y_{12})$. Now using 
  Lemma \ref{n-equiv}, we get that $f:(X, f^{-1}(Y_i))\to (Y, Y_i)$ is
  a \cwe\
  with respect to $f^{-1}(\wt Y_i)$ for $i=1,2$. Applying the converse of
  Lemma \ref{funda} we get that $f:(X, F_y)\to (Y,y)$ is a weak
  equivalence for all $y\in \wt Y_i$, $i=1,2$, and hence for
  all $y\in \wt Y_1\cup \wt Y_2$. This completes the proof of the
proposition.\end{proof}

\section{Proofs of the Theorems}

We begin with the proof of Theorem \ref{covering} which says that a 
\icq\ can be deduced from local data.

  \begin{proof}[Proof of Theorem \ref{covering}]
    The proof is a consequence of Lemma \ref{directlimit},
    Proposition \ref{coveringprop} and the Zorn's lemma.

Let ${\cal O}=\{U_{\alpha}\}_{\alpha\in J}$ be the open covering of $Y$
and $\{\wt U_{\alpha}\}_{\alpha\in J}$, be the exhaustive subsets of $\cal O$.
Recall that $\cup_{\alpha\in J}\wt U_{\alpha}$ is denoted by $\wt Y$.
Next, we define a set $\cal U$ as follows.
$${\cal U}=\{V_K\ |\ V_K\subset Y,\ V_K=
\cup_{\alpha\in K\subset J}U_{\alpha},\ \text{for some}\ K\subset J,\ V_K\
\text{is \icd\ with}$$
$$\text{respect to}\ \wt V_K:=\cup_{\alpha\in K} \wt U_{\alpha},\ \text{and for}\
K_1,K_2\subset J,\ V_{K_1}\cap V_{K_2}\ \text{is \icd\ 
  with}$$
$$\text{respect to}\ \cup_{\alpha\in K_1\cap K_2}\wt U_{\alpha} \}.$$

By hypothesis, $\cal U$ is nonempty, since it contains $\cal O$.
Now, we partially order $\cal U$ using set inclusions of its members.
Then, given any chain $\cal C$ in $\cal U$, Lemma \ref{directlimit} shows that
$\cup_{V\in {\cal C}}V$ is \icd\ with respect to
$\cup_{V\in {\cal C}}\wt V$.
This is because any compact subset of $\cup_{V\in {\cal C}}V$
is contained in one of the members of $\cal C$.

Therefore, by the Zorn's lemma $\cal U$ has a maximal element, say $V_{K_0}$. We claim
that $V_{K_0}=Y$. If not, then choose $U\in {\cal O}$ such that $U$ is not contained in
$V_{K_0}$. Now since, ${\cal O}\subset \cal U$, by definition of $\cal U$, $V_{K_0}\cap U$
is \icd\ with respect to $(\cup_{\alpha\in K_0}\wt U_{\alpha})\cap \wt U$.
Next, using Proposition \ref{coveringprop}, we get that $V_{K_0}\cup U$ is a \icd\
with respect to $(\cup_{\alpha\in K_0}\wt U_{\alpha})\cup \wt U$. Hence $V_{K_0}\cup U$
is an element of $\cal U$, which is a contradiction to the maximality of $V_{K_0}$.

Therefore, $V_{K_0}=Y$ and hence $f$ is a \icq\ with respect
to the exhaustive subset $\cup_{\alpha\in K_0}\wt U_{\alpha}$. This
completes the proof of the theorem.\end{proof}

Next, we prove Theorem \ref{limit} which says that under a colimit \kcq\ 
is preserved.

\begin{proof} [Proof of Theorem \ref{limit}]
     Theorem \ref{limit} is an immediate application of Lemma \ref{directlimit}. We
     just have to note that under the colimit topology on $Y$,
     induced by the filtration $\{Y_i\}_{i\in {\Bbb N}}$ of $T_1$ subspaces, any compact subset
     of $Y$ is contained in one of the members of the filtration.\end{proof}

   Recall that Theorem \ref{deformationthm} says that a certain
   kind of deformation preserves \icq.
   
   \begin{proof}[Proof of Theorem \ref{deformationthm}] Let us first recall the notations and
     the hypothesis.
     
     $f:X\to Y$ is a surjective map, $Y_1\subset Y$ and $X_1=f^{-1}(Y_1)$.
     $\wt Y\subset Y$ and $\wt Y_1\subset Y_1$ are 
     exhaustive subsets, such that $\wt Y_1\subset \wt Y$. $(H,h)$ is a deformation from $f$ to
     $f_1=f|_{X_1}:X_1\to Y_1$ such that $h_1(\wt Y)\subset \wt Y_1$.
     Furthermore, for all $y\in \wt Y$, $H_1|_{f^{-1}(y)}:f^{-1}(y)\to f^{-1}(h_1(y))$ is a
     weak equivalence.
     
     Now, consider Diagrams 10 and 11 for all $y\in \wt Y$.
       
       Note that, since $H_1|_{f^{-1}(y)}:f^{-1}(y)\to f^{-1}(h_1(y))$ is a weak equivalence, an
         application of the Five Lemma to the long exact sequence of
         homotopy groups of pairs proves that the composition $i\circ H_1$ is a
         weak equivalence. Next, note that since
         $H$ is a deformation, the inclusion map $i$ is a weak equivalence.
         Therefore, we conclude that the vertical map $H_1$ in
         Diagram 10 is a weak equivalence.

         \centerline{
           \xymatrix{(X, f^{-1}(y))\ar[d]^{H_1}\ar[rd]&\\
           (X_1, f^{-1}(h_1(y)))\ar[r]^{i}&(X, f^{-1}(h_1(y)))}}

       \centerline{Diagram 10}
       
         Hence, in Diagram 11, the top horizontal map is a weak equivalence.
         Also, the right hand side vertical map is a weak equivalence, since $Y_1$
         is \icd\ with respect to $\wt Y_1$. Finally, the below horizontal map
         is a weak equivalence since $h$ is a deformation.

         \centerline{
           \xymatrix{(X, f^{-1}(y))\ar[r]^{\!\!\!\! \!\! H_1}\ar[d]^f&(X_1, f^{-1}(h_1(y)))\ar[d]^{f_1}\\
             (Y, y)\ar[r]^{h_1}&(Y_1, h_1(y))}}

         \centerline{Diagram 11}
         
             Therefore, we conclude that the left hand side vertical map in Diagram 11 is a
         weak equivalence for all $y\in \wt Y$, that is $f$ is a \icq\ with respect to $\wt Y$.

         This completes the proof of the theorem.\end{proof}

       \newpage
\bibliographystyle{plain}

\begin{thebibliography}{60}

\bibitem{DT58}
  Dold, A., Thom, R.: 
  \newblock Quasifaserungen und unendliche symmetrische Produkte.
  \newblock  Ann. of Math., Second Series {\bf 67}, 239-281 (1958)

  \bibitem{Hat03}
  Hatcher, A.:
  \newblock Algebraic Topology.
  \newblock Cambridge University Press. First south Asian edition. (2003)

\bibitem{Hem76}
  Hempel, J.:
  \newblock $3$\hyp{}manifolds.
  \newblock Ann. of Math. Studies 86. Princeton University Press, Princeton, N.J. (1976)
  
\bibitem{May88}
  May, J.P.:
  \newblock Weak equivalences and quasifibrations.
  \newblock In: Groups of Self-equivalences and Related Topics, pp. 91-101. Lecture
  Notes in Mathematics, {\bf 1425}. Eds. A. Dold, B. Eckmann and F. Takens. Springer-Verlag,
  Proceedings, Montreal (1988)

\bibitem{Moe02}
  I. Moerdijk,
  \newblock Orbifolds as groupoids: an introduction,
  \newblock {\em Orbifolds in mathematics and physics} (Madison, WI, 2001), 205-222,
  \newblock Contemp. Math., 310, {\it Amer. Math. Soc., Providence, RI}, (2002)
  
\bibitem{Rou20}
Roushon, S.K.: 
  \newblock Configuration Lie groupoids and orbifold braid groups. 
  \newblock Bull. Sci. math. {\bf 171} (2021) 35p.
  \newblock doi:10.1016/j.bulsci.2021.103028

\bibitem{Rou21}
  Roushon, S.K.: 
  \newblock Quasifibrations in configuration Lie groupoids and
  orbifold braid groups. Preprint.
  \newblock https://doi.org/10.48550/arXiv:2106.08110

\bibitem{Rou22}
  Roushon, S.K.:
  \newblock Almost\hyp{}quasifibrations and fundamental groups of orbit configuration
  spaces.
  \newblock https://doi.org/10.48550/arXiv:2111.06159

  \bibitem{Sco83}
    P. Scott,
    \newblock The geometries of 3-manifolds,
    \newblock {\it Bull. London Math. Soc.} 15 no. 5, 401-487 (1983)
  
\end{thebibliography}
\ifx\undefined\bysame
\newcommand{\bysame}{\leavevmode\hbox to3em{\hrulefill},}
\fi

\end{document}